\documentclass[12pt,oneside,reqno]{amsart}
\usepackage{amsmath,amsfonts,amsthm,amssymb,amscd,url}
\usepackage{dsfont}
\usepackage{enumerate}
\usepackage[margin=1in]{geometry}\usepackage{multirow,booktabs}
\usepackage{bigints}\textwidth 13.5cm\DeclareFontFamily{OT1}{rsfs}{}\DeclareFontShape{OT1}{rsfs}{n}{it}{<-> rsfs10}{}\DeclareMathAlphabet{\mathscr}{OT1}{rsfs}{n}{it}\newtheorem{prop}{Proposition}[section]\newtheorem{theorem}[prop]{Theorem}\newtheorem{corollary}[prop]{Corollary}\theoremstyle{theorem}\newtheorem*{AH}{The Alternative Hypothesis (AH)}\newtheorem{lemma}[prop]{Lemma}\newtheorem*{defn*}{Definition}\theoremstyle{definition}
\newtheorem{Rem}{Remark}
\numberwithin{equation}{section}\author{Farzad Aryan}
\title{Higher moments of distribution of zeta zeros}
\begin{document}
\maketitle
\begin{abstract}We develop a method for mean-value  estimation of long Dirichlet polynomials. For an application, we use our method to study properties of  the logarithmic derivative of the Riemann zeta function. 
\end{abstract}

\section{\bf Introduction.}
In this paper we introduce a method to bound the moments of a long Dirichlet polynomial. Let \begin{equation*} \mathcal{B}(s)= \sum\frac{b(n)}{n^{s}}, \end{equation*}  be of length larger than $T^{1+ \tau},$ for $\tau>0$. We are generally interested to estimate \begin{equation} \label{longshifted}     \int \omega(\tfrac{1}{2} +it)|\mathcal{B}(\tfrac{1}{2} +it)|^k, \end{equation} where $\omega$ is  a smooth function that is supported on $t \sim T$. Throughout the paper we take 
\begin{equation}
 \omega(\tfrac{1}{2} + it)=\frac{\log T}{\sqrt{\pi} T}\hspace{1 mm}\displaystyle{ e^{-\frac{(t-T)^2
\log^2 T}{T^2}}}.
\end{equation}

Many important problems in analytic number theory, can be transform into estimating mean-values like \eqref{longshifted}. In case that we do not have a clear understanding of $$\sum_{n<x}b(n)b(n+h),$$ evaluating \eqref{longshifted} turns into a difficult task.\\

 We use decomposition of $\mathcal{B}(s)$ and some ideas from the probability theory to get lower bounds on some families of interesting sequences. Our first corollary is an example of mean-value estimates that our method can handle. Up to the author's knowledge, this is the first result that provides a lower bound on mean-value of Dirichlet polynomials involving Liouvill's function.
\begin{corollary}
\label{CorMain}
Let $d(n) $ be the divisor function and $\lambda(n)$ be Liouvill's function. Define $$b(n):= -\frac{2}{ \log T} \sum_{\substack{n=mp^j \\ m< T\log^{-2} T \\ p^j<T   }}\log p\big(1-\frac{\log p^j}{ \log T}\big) \lambda(m) d(m).$$ If all of the divisor of $n$ are smaller than $T,$ then $b(n)$ approximately  equals to $ \lambda(n)d(n)\frac{ \log n}{ \log T}.$ We have that 
 \begin{equation*}
     \int \omega(\tfrac{1}{2}+it)\bigg|\displaystyle{ \sum_{ n \geq \frac{T}{\log^{2}T}}} \hspace{2 mm} \displaystyle{ \frac{b(n)}{n^{\tfrac{_1}{2}  + it}}}\bigg|^2dt > 0.009 A_2 \log^{4} T,
 \end{equation*}
where $A_2$ is a constant. 
\end{corollary}
Finding an asymptotic for the above requires \cite{Gold-Gon} assuming very strong bound on the shifted convolution sum of $\lambda(n)d(n)$ which falls more into the premises of Sarnak conjecture \cite{Sarnak} with power savings. Our lower bound is of the correct order of magnitude, with what we get if we assume an strong version of Sarnak conjecture. Therefore our result, can be viewed as an evidence to support the conjecture.\\

 Matom\"{a}ki and Radziwiłł \cite{MotRad} proved an upper bound for a similar mean-values with $b(n)=\lambda(n).$ Their work resulted in many interesting applications including answering the Erdos discrepancy problem by Tao. 
\\

The paper is organized as follows: \\

 First we explain the mollifier method that is often used to study properties of $L$-functions. Next we state our corollaries to show the improvements we obtain  over this method. In section \ref{mainthm} we state our main theorem. In section \ref{generalway} we explain how the method can be generally apply to obtain lower bound on \eqref{longshifted}. In section \ref{LSzeros} we explain the connection of our result to the Landau-Siegel zero problem. In Section \ref{sec-prob} we give example of random variables that their moments are related to moments we are seeking to estimate. In the last section we give proofs of our results. 
 
\subsection{The mollifier method}
Using mollifiers to study analytic properties of $L$-functions is  a common tool in number theory.  Generally speaking if we would like to show that an $L$-function has a certain property, we can use the following recipe:\\

\begin{enumerate}
    \item First, using a Dirichlet polynomial $\displaystyle{A(s)= \sum a(n)n^{-s}}$, that highlights that property, we build a probability measure in the following  way. Define,   
    \begin{equation}
    \label{defmeasur}
\mu_A((a,b]):= \frac{\int_{a}^{b} \omega(\frac{_1}{2}+it)\big|A(\frac{_1}{2}+it)\big|^2}{
\int \omega(\frac{_1}{2}+it)\big|A(\frac{_1}{2}+it)\big|^2}.
    \end{equation}
The polynomial  $A$ is commonly referred to as a mollifier. If we set $A(s)=1,$ then $\mu_1$ gives the probability measure we get out of Lebesgue integration.\\

\item Next we consider a series expansion of the $L$-function, denoted by $\mathcal{L}$, and estimate the first moment with respect to $\mu_A$:
\begin{equation}
\label{1stmom}
    \int \mathcal{L} d\mu_{A}.
\end{equation}
The result is an indicator of the degree that the $L$-function satisfy the property. For example proving \eqref{1stmom} is bigger than $\overline{\mathcal{L}}$ shows that there exist $t$ such that $|L(\frac{_1}{2}+it)| > \overline{\mathcal{L}}.$ Going forward, we refer to the above as the mollifier method.\\
\end{enumerate} 

This method has been applied in~\cite{Sound, BSe} to the problem of large values of the Riemann zeta function\footnote{In the context of large values of Riemann zeta function the method is called the resonance method.} and in \cite{Me2} to the problem of gaps between zeta zeros. \\

In this paper we show how to extract more information out of the mollifier method, by engaging in higher moments estimation. Our results are an indicator of improvement we can obtain. Let us consider an smooth series expansion of the logarithmic derivative of zeta ($\displaystyle{\tfrac{\zeta'}{\zeta}}(s))$:

\begin{equation}\label{Z-def}\mathcal{Z}_\alpha(s):=\frac{-2}{\alpha \log T} \sum_{n<T^{\alpha}} \frac{\Lambda(n)}{n^{s}} \Big(1-\frac{\log n}{\alpha \log T}\Big),
\end{equation} 

we write \begin{equation}
\label{ReIm}\mathcal{Z}_\alpha(\tfrac{1}{2} + it)= \emph{\emph{C}}_\alpha(t)+ i \emph{\emph{Im}}_\alpha(t).
\end{equation} 
Understanding the behaviour of $\mathcal{Z}_\alpha$ has significant implications. For example, showing that there exist a $t \in \mathbb{R}$ such that  $$\emph{\emph{C}}_1(t)> 1 \text{ or } \emph{\emph{C}}_2(t)> 0.5,$$ proves that there are no Landau-Siegel zeros.\footnote{For a full discussion on the connection with Landau-Siegel zeros see the introduction and Lemma 8.1 in \cite{Me2}.}  Assuming the Riemann hypothesis (RH) $\emph{\emph{C}}_\alpha(t)$ is large when $t$ is close to a zero cluster and it is approximately $-\tfrac{1}{\alpha}$ on large gaps (see \eqref{calpha}).   \\

Let $\lambda$ be Liouvill's function. In \cite{Me2}, using the Mollifer method we proved that
\begin{equation}
\label{resu21}
    \frac{1}{\log T} \int \omega(\tfrac{1}{2} +it)\hspace{0.5 mm}\emph{C}_1(t) \hspace{0.5 mm} \bigg| \sum_{n< T\log^{-2} T} \hspace{2mm} \frac{\lambda(n)}{n^{\frac{_1}{2} + it}}\bigg|^2dt = \frac{2}{3} + O(\tfrac{1}{\log T}).
\end{equation} 
For simplicity, and by using the notation in \eqref{defmeasur} we put $\mu_\lambda$\footnote{This can be viewed as an approximation of $\zeta(2s)/\zeta(s)$. Bulk of $\mu_\lambda$ is distributed in the close vicinity of zeta zeros. See section 7 in \cite{Me2}.} in place of the mollifier $$\omega(\tfrac{1}{2}+it)\frac{1}{\log T}|\sum \lambda(n)n^{-{1}/{2}-it}|^2.$$  Using this we can write \eqref{resu21}  in the following way: 
\begin{equation}
\label{resu22}
   \int  \emph{\emph{C}}_1(t) d\mu_\lambda = \tfrac{2}{3} + O(\tfrac{1}{\log T}).
\end{equation}

We proceed by applying our method to study the higher moments of $\emph{\emph{C}}_\alpha$ and $\emph{\emph{Im}}_\alpha$. A heuristic given in \cite{Harper} suggest that $\mathcal{Z}_\alpha$ should behave like a sum of independent random variables since the primes are multiplicatively independent. Therefore it is expected that value distribution of
objects like $\mathcal{Z}_\alpha$ 
behave like normal distribution. When $\alpha< 1/k$ we can estimate $k$-th moment with respect to the Lebesgue measures. For $\alpha \geq 1,$ estimating higher moments with respect to the Lebesgue measure is out of reach, let alone twisted higher moments with $\mu_A$. As our second corollary we show: 
\begin{corollary} \label{cor 01}
By the same notation as \eqref{ReIm} and letting $\mu_\lambda$ be as \eqref{defmeasur} with $a(n)= \lambda(n)$,  we have that
\begin{align}  
\label{eq1.8}
  & \int  {\emph{C}}^{2}_1(t) d\mu_\lambda \geq  0.46666, \\ & 
   \int  {\emph{C}}^{2}_2(t) d\mu_\lambda \geq  0.174998. 
 \end{align}
\end{corollary}
To see the improvement upon the mollifier method, note that  using the corollary we can show that there exist $t \in \text{ Supp}( \mu_\lambda)$ such that $$|\emph{\emph{C}}_1(t) |> 0.6831,$$ which obviously is an improvement over $2/3 \simeq 0.666$ that we get from \eqref{resu22}. Only applying \eqref{resu22} and the fact that the variance of a real random variable is greater than its mean square, gives us the lower bound $4/9 \simeq 0.444,$ for the RHS of \eqref{eq1.8}. \\   

 The imaginary part of $\mathcal{Z}_\alpha$ is more mysterious than its real part due to its oscillatory nature. Therefore, proving results similar to \eqref{resu21} or Corollary \ref{cor 01} is much trickier. For example,  we have $\int \omega(\tfrac{1}{2} + it) \emph{\emph{Im}}_\alpha(t) d\mu_\lambda \sim 0$, which gives nothing but the trivial bound on the variance of $\emph{\emph{Im}}_\alpha$.  In terms of large values of the imaginary part we can apply Landau-Gonek formula~\cite{Gonek}. For $\alpha<1$, and letting $N(T)$ denote the number of zeta zeros with height between $0$ and $T,$ we obtain
\begin{equation}
 \label{LGform}
     \frac{1}{N(T)} \sum_{0< \gamma < T}\displaystyle{\emph{\emph{Im}}_{1-\epsilon}}(\gamma + \tfrac{2\pi d}{\log T}) \sim \int_{0}^{1} 2u(1-u) \sin(2\pi u d)du. 
 \end{equation}
 The integral has the maximum $\simeq 0.27$ at the shift $d \simeq 0.4147.$ Therefore, average value of ${\emph{\emph{Im}}_{1-\epsilon}},$ shifted slightly to the right of zeros, is about $0.27.$ \\
 
 This observation suggest that the $\emph{\emph{Im}}_\alpha$ has large values when we shift slightly to the right or left of a zero. In the next section we discuss the importance of $\emph{\emph{Im}}_\alpha$ and we state our results on the large values  of this function.\\
 
\subsection{ The imaginary part of the logarithmic derivative of $\zeta$} Studying the imaginary part of $\mathcal{Z}_\alpha$ is an important subject in part, due to its connection to the  error term in the formula for the number of zeros of the Riemann zeta function.\\

Let $S(t)$ be the deviation from the expected value in the zeta zeros counting function, which is formally defined as  $$S(t)= \frac{1}{\pi} \arg \zeta(\tfrac{1}{2}+it).$$ A famous theorem of Selberg states that on average the value of $S(t)$ is related to the imaginary part of the logarithmic derivative of zeta:
\begin{theorem}[Selberg] For $k\geq 1,$ we have 
\begin{equation}\label{selsoft} \int_{0}^{T} \Big|S(t)+ \frac{1}{\pi} \hspace{0.25 mm} \displaystyle{\Im \sum_{n< T^{1/k}} \hspace{2 mm}\frac{\Lambda(n)}{n^{\tfrac{_1}{2}+it}\small{ \log n}}}\Big|^k \ll_k T. \end{equation}
\end{theorem} Using the above, he went on to show that all even moments of $S(t)$ match the moments of the Gaussian distribution, which confirm the speculation that $S(t)$ should behave like a sum of independent RV. \\

Similar to Corollary \ref{cor 01} we give a lower bounds on variance of $\emph{\emph{Im}}_\alpha$ with respect to $\mu_\lambda.$ Also, we will introduce a surprisingly interesting measure that significantly improve the bounds we get from $\mu_\lambda$ on the large values of $\emph{\emph{Im}}_\alpha$. 

\begin{corollary} 
\label{cor 1}
By the same notation as \eqref{ReIm}, we have that
\begin{align} & \label{one}  \int  {\emph{Im}_1^2}(t) \hspace{0.5 mm} d\mu_\lambda  \geq 0.1, \\ & \label{two}  \int  {\emph{Im}_1^2}(t) \emph{C}_1^2(t)  d\mu_\lambda \geq 0.02763 \cdots,
\end{align}
 \end{corollary}
 \begin{Rem}
 An important feature of Corollaries \ref{cor 01} and \ref{cor 1} is that asymptotically estimating \eqref{one} and \eqref{two} requires off-diagonal treatments that is far from the reach of our current techniques. At minimum it requires a strong version of the Chowla conjecture plus a power saving in shifted convolution of von Mangoldt and Liouville's function.
 \end{Rem}
 In section \ref{LSzeros}, we will show that the imaginary part will also play a roll in the problem of gaps between zeta zeros.   
\subsection{An interesting arithmetic function \& large values of $\emph{\emph{Im}}_\alpha$} We define a generalization of the Liouville's function as follows:  \\

Let $n=p^{\alpha_1}_1 \cdots p^{\alpha_m}_m,$ and define 
\begin{equation}
\label{lambda2}
 \lambda_k(n)= (-1)^{ [\frac{\Omega(n)}{k} ]},
\end{equation}
where $\Omega(n)= \alpha_1 + \cdots + \alpha_k,$ and $[\cdot]$  is the  floor function. \\

\noindent For $k=1$, $\lambda_1$ is the Liouville's $\lambda$-function.\\
For $k=2$,  we have $\lambda_2(1)= \lambda_2(p_1)=1, \hspace{2 mm} \lambda_2(p_1p_2)=\lambda_2(p_1p_2p_3)=-1 \hspace{2 mm}$ \ldots
\\

To our knowledge no connection has been made between 
\begin{equation}
\label{Llambda2}
L(s, \lambda_2):= \sum_{n=1}^{\infty} \frac{\lambda_2(n)}{n^s}, \hspace{5 mm} \Re(s)>1,
\end{equation}   and the zeros of the Riemann zeta function. Here we make an interesting link between  $L(s, \lambda_2),$ and large values of $\emph{\emph{Im}}_\alpha$ and gaps between zeta zeros. We will show \\
\begin{enumerate}
    \item  The mollifier $\mu_{\lambda_{2}}$ that we get by truncating $L(s, \lambda_2)$ for $n< T \log^{-2} T$ correlates highly with large values of $\emph{\emph{Im}}_\alpha.$ 
    \item Support of $\mu_{\lambda_{2}}$ is expected to be on mid size gaps, and away from close vicinity of zeros. 
\end{enumerate} 
To establish these facts more precisely, first we put forward this corollary:
 
\begin{corollary} Let $\mu_\lambda$ be as \eqref{defmeasur} with $a(n)= \lambda(n)$  and define $\mu_\zeta$ similarly by setting $a(n)= 1,$ for $n<n< T \log^{-2} T$. We have that 
\begin{align} \label{c2-Im2}
    \int \big(\emph{\emph{C}}^2_1(t) - {\emph{Im}}^2_1(t)\big) d\mu_\lambda= \int \big(\emph{\emph{C}}^2_1(t) - {\emph{Im}}^2_1(t)\big) d\mu_\zeta \sim  0.36666\cdots 
\end{align}
    \end{corollary}
This corollary shows that $\emph{\emph{C}}_1^2$ is larger than $\emph{\emph{Im}}_1^2$ on supports of both $\mu_\lambda$ and $\mu_\zeta$. In \cite{Me2} we showed that bulk of $\mu_\lambda$ is centred close to zeros, and bulk of $\mu_\zeta$ is centred away from zeros and on large gaps.\footnote{ A similar result for $\emph{\emph{C}}_2$ and $\emph{\emph{Im}}_2$ holds with the RHS of \eqref{c2-Im2} replaced with $0.12708\cdots$.} \\

Now we state the main corollary of this section:
\begin{corollary} \label{corlambda2}Let $\lambda_2$ be as \eqref{lambda2} and $\mu_{\lambda_2}$ be a measure we build with $\lambda_2$ as in \eqref{defmeasur}. We have that 
\begin{equation*}
\int  \emph{Im}^{2}_1(t) d\mu_{\lambda_2} \geq  0.46666.
\end{equation*}
Moreover, $$\int \big(  {\emph{Im}}^2_1(t)- \emph{\emph{C}}^2_1(t)\big) d\mu_{\lambda_2} =  0.36666\cdots $$

\end{corollary}
Similar to corollary  \ref{cor 1} this shows that there exist a $t$ in the support of $\mu_{\lambda_{2}}$ such that $$|\emph{\emph{Im}}_1(t) |> 0.6831.$$ Also, we have that the imaginary part is larger than the real part on the support of $\mu_{\lambda_2}.$ This means that this measure is mainly supported on region where $\mu_\lambda$ and  $\mu_\zeta$ are almost void. Therefore we expect that $\mu_{\lambda_{2}}$ be supported on mid size gaps, and away from close vicinity of zeros. \\

As we mentioned before $\mu_\lambda$ is an approximation of the behaviour of  $|\tfrac{\zeta(2s)}{\zeta(s)}|^2$, and consequently it  make sense that  
it is mainly centred in a close vicinity of zeros. It is not clear how we can have a similar explanation for $\mu_{\lambda_{2}}$. The alteration that occurs in the definition of $\lambda_2$ is an arithmetic one. It would be interesting to have an explanation as to how the analytic shift happens in the support of $\mu_{\lambda_2}$.

\section{\bf The Main Theorem} \label{mainthm}
 Our main theorem gives all of the pseudo moments of $\mathcal{Z}_\alpha$ with respect to variety of measures built in \eqref{defmeasur} with 
 \begin{equation}
 \label{def_an}a(n)= \lambda_{\beta_1, \beta_2}(n)d_{r}(n),
 \end{equation}
 where $d_r$ is the generalized divisor function and $\lambda_{\beta_1, \beta_2}$ is defined to be a completely multiplicative function with  $\lambda_{\beta_1, \beta_2}(p)=-1$ for $ T^{\beta_1}< p \leq T^{\beta_2}$ and $\lambda_{\beta_1, \beta_2}(p)=1$ otherwise.

\begin{theorem}
\label{Thm}
Let $\mathcal{Z}_\alpha$ be as in \eqref{Z-def} and $\mu_{A_{\beta_1, \beta_2, r }}$ be the measure we built with $a(n)= \lambda_{\beta_1, \beta_2}(n)d_{r}(n)$ as in \eqref{defmeasur}. We have
 \begin{align} 
 \label{Z-mom}
 \int & \mathcal{Z}^{k}_\alpha(\tfrac{1}{2}+it) d\mu_{A_{\beta_1, \beta_2, r }} +O(\frac{1}{\log T})\\ & \notag= \big(\frac{-2r}{\alpha}\big)^k  \frac{ \small \displaystyle{  \int_{0}^{1} \ldots \int_{0}^{\small 1- v_2  \cdots- v_k }  \Lambda_{\beta_1, \beta_2}(v_1) \ldots \Lambda_{\beta_1, \beta_2}(v_k)  (1-v_1\cdots -v_k)^{r^2} \small{d_{v_1} \ldots d_{v_k}} }}{\int_{0}^{1} v^{r^2-1}(1-v)^{2\eta}dv}
 \end{align}
 where $$
\Lambda_{\beta_1, \beta_2}(v)=
\begin{cases}
- \big(1-\tfrac{v}{\alpha}\big)  \text{ for } \beta_1 <u \leq \beta_2 \\
1-\tfrac{v}{\alpha}, \text{ otherwise.} \\
\end{cases}
$$
 For the special case $a(n)= \lambda(n)$ and $\alpha\geq 1$ we have 
\begin{equation}
\label{high-mom} \int \mathcal{Z}^{k}_\alpha(\tfrac{1}{2}+it) d\mu_\lambda(t)= \Big(\frac{2}{\alpha}\Big)^k \sum_{i=0}^{k} {k \choose i} \big(\frac{-1}{\alpha}\big)^i \frac{1}{(k+i+1)!}+O(\frac{1}{\log T}).
\end{equation}
Moreover for $a(n)= \lambda_2(n)$, and even moments we have that 
\begin{equation}
\label{thmlambda2}
     \int \mathcal{Z}^{2k}_\alpha(\tfrac{1}{2}+it) d\mu_{\lambda_2}(t) =  (-1)^k\int \mathcal{Z}^{2k}_\alpha(\tfrac{1}{2}+it) d\mu_\lambda(t).
\end{equation}
\end{theorem}
\begin{Rem} An interesting feature of the theorem is that the $k$-th moment decays almost like $1/(k+1)!$ times $(2/\alpha)^k.$ This is the property  that helps us to get the mentioned results on the imaginary part. It would be interesting to categorize random variables whose moments satisfy this property. We will discuss an example of this in  Section \ref{sec-prob}. 
\end{Rem}

Dirichlet polynomial $\mathcal{Z}_\alpha$ is a  smooth approximation of $\frac{\zeta'}{\zeta}$ using the the triangular function $(1- \tfrac{\cdot}{{\alpha \log T}})$ as a weight. Using this smooth weight is not necessary for our results, but it is very useful for investigating the distribution of zeros of the zeta function. 
\\

Let $\sigma= \tfrac{1}{2}+ \tfrac{\beta}{\log T}$ where $\beta \gg 1.$ Assuming the Riemann hypothesis Selberg proved\begin{equation}\label{Sel2}\frac{1}{T} \int_{0}^{T} \Big| \frac{1}{\log T} \frac{\zeta'}{\zeta}(\sigma+it) \Big|^2dt \ll_\beta 1.\end{equation}Later Goldston, Gonek and Montgomery\cite{GGM} proved that getting an asymptotic for \eqref{Sel2} is equivalent to Montgomery's pair correlation conjecture on the distribution of the spacing between zeta zeros. \\

In the proof they use the simple fact that 
\begin{equation}
\label{GGMeq}
 \Big|\frac{\zeta'}{\zeta}(s)\Big|^2 = 2\Re \Big( \frac{\zeta'}{\zeta}(s)\Big)^2- \Re\Big(\big(\frac{\zeta'}{\zeta}(s)\big)^2\Big).
\end{equation}
Then they integrate the LHS of \eqref{GGMeq} with respect to the Lebesgue measure. Consequently the last term in RHS of the \eqref{GGMeq} disappears and only 
$$\Re \Big( \frac{\zeta'}{\zeta}(s)\Big)^2$$ plays a role. As we mentioned the real part is closely tied to the distribution of zeros. Therefore, they used Montgomery's study of the pair correlation to get their results. \\

In our work by switching the measure from the Lebesgue, we take most advantage of $$\Re\Big(\big(\frac{\zeta'}{\zeta}(s)\big)^2\Big),$$  in the RHS of \eqref{GGMeq}, which does not contribute when we use the Lebesgue measure. In addition, we show that the imaginary part also play a role in the problem of gaps between zeta zeros. \\

We conclude this section with stating our result related to the mean-square of $\frac{\zeta'}{\zeta}$. In the next section we explain how the method of this paper can generally be applied to obtain a lower bound on the mean-square of general Dirichlet polynomials. 
\begin{corollary} 
\label{Cor14} 
For $\alpha \geq 1,$ we have that 
\begin{equation*}
     \int \big|\mathcal{Z}_\alpha(\tfrac{1}{2}+it)  \big|^2 d\mu_\lambda \geq \frac{4}{3\alpha^2} - \frac{1}{\alpha^3} + \frac{7}{30\alpha^4}.
     \end{equation*}
\end{corollary}
\section{\bf General Way of Applying our Method} \label{generalway} Let $\mathcal{Z}= \mathcal{C} + iIm,$ be a complex random variable.  If we can determine all the moments of $|\mathcal{Z}|$ with respect to a probability measure, i.e.
\begin{equation}
\label{absV}
    \int |\mathcal{Z}|^k d\mu,
\end{equation}
then we can determine the distribution. The question is that how much information we can get from pseudo moments of $\mathcal{Z},$ i.e.
\begin{equation}
 \label{Pseudo}  
\mathcal{M}_k:=\int \mathcal{Z}^k d\mu?
\end{equation}
The mollifier method that we discussed in the introduction, in essence, is looking at the above for $k=1.$\\


When we work with $L$-functions, $\mathcal{Z}$ can be an approximation of our function and $\mu$ can be the appropriate probability measure built by a Dirichlet polynomial. Usually if the first pseudo moment is calculable, then the higher moments are as well. Our method works the best, and yield improvements over the mollifer method when the $k$-th pseudo moment is different than the first moment to the power of $k$. \\

To demonstrate this, suppose that all pseudo moments are real. Moreover, assume 
\begin{equation}
\label{trivialcase}
E[\mathcal{Z}^k]= \mathcal{M}_k = (\mathcal{M}_1)^k + \delta_k, \textrm{ with } \hspace{1 mm} |\delta_k|>0.
\end{equation}

Since we considered $\mu$ to be a probability measure, using the fact that the variance of a real random variable is greater than its mean square, we have
\begin{equation}
    \label{triVa}
E[\Re (\mathcal{Z})^2] > (\mathcal{M}_1)^2.
\end{equation} Also, by using the assumption that moments are real and \eqref{trivialcase}, we have 
\begin{align}
\label{2ndm}
    E[\Re( \mathcal{Z}^2)]= E[(\Re\mathcal{Z})^2]- E[(\Im \mathcal{Z})^2]= (\mathcal{M}_1)^2 + \delta_2.
\end{align} 
If $\delta_2>0,$ then \eqref{2ndm} immediately gives $E[(\Re\mathcal{Z})^2]> (\mathcal{M}_1)^2 + \delta_2,$ which obviously is better than \eqref{triVa}. If  $\delta_2<0,$ then 
\begin{align}
\notag
 E[(\Im \mathcal{Z})^2]= E[(\Re\mathcal{Z})^2] -(\mathcal{M}_1)^2 - \delta_2 > |\delta_2|,
\end{align} Which gives a lower bound on the square of the imaginary part of $\mathcal{Z}$ and hence, a lower bound on $|\mathcal{Z}|^2:$
$$\int |\mathcal{Z}|^2 d\mu > \mathcal{M}^2_1+ | \delta_2|.$$


\subsection{Mean-value of long Dirichlet polynomials} For application to number theory, imagine that we have a Dirichlet polynomials of length larger than $T^{1+ \tau},$ for $\tau>0$:
\begin{equation*}
\mathcal{B}(s)= \sum\frac{b(n)}{n^{s}},
\end{equation*}
and we are interested in estimating
\begin{equation}
\label{decomp}
    \int \omega(\tfrac{1}{2} +it)|\mathcal{B}(\tfrac{1}{2} +it)|^2.
\end{equation}
Recall that the support  $\omega$ is on an interval of length  almost $T$ and therefore to estimate  \eqref{decomp}, normally we need to engage in off-diagonal estimation that requires some knowledge on the shifted convolution sums of $b(n)$ \footnote{See \cite{Gold-Gon} for results on mean-value theorems for long Dirichlet polynomials.}. In many cases knowing about this turns to a harder problem than our original one in estimating \eqref{decomp}.\\

Now assume that we can decompose $\mathcal{B}$ as a product of two Dirichlet polynomial plus another polynomial with a negligible contribution:
\begin{equation*}
    \mathcal{B}(s)= \mathcal{G}(s)\mathcal{F}(s) + \mathcal{E}(s).
\end{equation*}
Using the triangle inequity we have that $ |\mathcal{B}|^2 \geq  |\mathcal{G}\mathcal{F}|^2 - |\mathcal{E}|^2.$\\

In this decomposition the length of $\mathcal{F}$ should be smaller than $T^{1-\epsilon}$. Next we turn $\mathcal{F}$ to a probability measure $\mu_{\mathcal{F}}$ by dividing with $$\int \omega(\tfrac{1}{2} +it)|\mathcal{F}(\tfrac{1}{2} +it)|^2,$$ which is calculable using Montgomery and Vaughan method~\cite{MV}.\\

Again by using the fact that length of $\mathcal{F}$ is smaller than $T^{1-\epsilon}$ we are able to estimate pseudo moments of $\mathcal{G}$: $$ \varrho_k =\int \mathcal{G}^k d\mu_{\mathcal{F}}.$$

By using the mollifier method we can have that 
\begin{equation}
\label{moligen}
        \int \omega(\tfrac{1}{2} +it)|\mathcal{B}(\tfrac{1}{2} +it)|^2 \geq \varrho^2_1 \int \omega(\tfrac{1}{2} +it)|\mathcal{F}(\tfrac{1}{2} +it)|^2,
\end{equation}
assuming that $\int |\mathcal{E}|^2 d\mu_{\mathcal{F}}$ has a negligible contribution. \\

Now as we explained in \eqref{trivialcase}, \eqref{triVa}, and \eqref{2ndm} if for $\sigma \neq 0$ we have that $$\varrho_2 - \varrho^2_1 = \sigma,$$
then we can get a better lower bound than \eqref{moligen}:
\begin{equation*}    
\int \omega(\tfrac{1}{2} +it)|\mathcal{B}(\tfrac{1}{2} +it)|^2 \geq (\varrho^2_1+ |\sigma| ) \int \omega(\tfrac{1}{2} +it)|\mathcal{F}(\tfrac{1}{2} +it)|^2.
\end{equation*}
\\


Next we move to the problem of gaps between the zeros of the Riemann zeta function.
\section{\bf Landau-Siegel Zeros and the Distribution of Zeta Zeros} \label{LSzeros} Distribution of zeros of the Riemann the zeta function on the critical line, is an important problem, in part, due to its connection to Landau-Siegel zeros and their effect on the class number formula. As we explained extensively in the introduction of~\cite{Me2}, in order to reject the possibility of Landau-Siegel zeros we need to rule out a certain distribution of zeta zeros known as the alternative hypothesis (AH).  \\



Let $\bar\gamma = \frac{1}{2\pi} \gamma \log(\frac{\gamma}{2\pi})$ be the normalized equivalent of an imaginary part of a zeta zero. 
By $ \gamma^+$
 and$ \gamma^-$ we mean the zero after and before $\gamma$ respectively.
 
\begin{AH} There exists a real number $T_0$ such that if $\gamma > T_0$ , then$$\tilde{\gamma}^{+}- \tilde{\gamma} \in \tfrac{1}{2}\mathbb{Z}.$$That is, almost all the normalized neighbour spacing's are an integer or half-integer.\\\end{AH}
This hypothesis is believed by many to be ``absurd," however no one yet knows how to reject it. Commonly believed conjecture regarding the distribution of zeta zeros is the Montgomery's pair correlation conjecture (PCC). This conjecture suggest that we will see a certain randomness in the distribution of zeta zeros as oppose to the rigidity that we see in AH. For example under PCC and RH we expect that $\emph{\emph{C}}_\alpha(t)$ must get arbitrary large. To justify this note that we have a neat description of $\emph{\emph{C}}_\alpha(t)$ that connects it to zeros in the vicinity of $t$: \\

For $\alpha \leq 4$ and $t \sim T,$ define \footnote{We need to consider the complex definition of sine function in case that RH does not hold.}: 
\begin{equation}
\label{calpha}
\mathcal{C}_{\alpha}(t)= \sum_{\zeta(\rho)=0} \bigg(\frac{\sin(\tfrac{\alpha}{2i}(\rho-(\tfrac{1}{2}+it))\log T)}{\tfrac{\alpha}{2i}(\rho-(\tfrac{1}{2}+it))\log T}\bigg)^2-\alpha^{-1}.
\end{equation} 
Using a formula for expansion of $\mathcal{C}_{\alpha}$ (see \cite[ proof of Lemma 1]{Gonek}) we have that $$\mathcal{C}_{\alpha}(t)= \emph{\emph{C}}_\alpha(t)+ O(\tfrac{1}{\log T}).$$
Assuming RH, \eqref{calpha} turns into 
\begin{equation}
\label{calpha1}
\mathcal{C}_{\alpha}(t)= \sum_{\zeta(\rho)=0} \bigg(\frac{\sin(\tfrac{\alpha}{2}(\gamma-t)\log T)}{\tfrac{\alpha}{2}(\gamma-t)\log T)}\bigg)^2-\alpha^{-1},
\end{equation} 
and if there are many zeros very close to $t,$ they each add a positive amount to the sum above, which increase the value of $\emph{\emph{C}}_\alpha(t)$. \\

In order to look more closely into AH we will take $\alpha=2$ in \eqref{calpha1}. The reason behind it, is that $\mathcal{Z}_2$ captures the essence of AH very well. To see this note, under RH, proving $\Re \mathcal{Z}_2(t)> 1/2,$ for some $t$, would precisely reject AH by showing that there must be gaps of size not equal to half integers. For example, consider \eqref{calpha1} for $\rho= \tfrac{1}{2}+i\gamma$ and $t= \gamma.$ Moreover assume that $\tilde{\gamma}^{+}- \tilde{\gamma}=0.75$, which immediately gives $\emph{\emph{C}}_\alpha(\gamma)> 0.545$. Thus, if, for instance, one can show that 
\begin{equation}
\label{-AH}
   \frac{2\pi}{T\log T} \sum_{0< \gamma < T} \Re \mathcal{Z}_2(\tfrac{1}{2}+i\gamma)> 1/2, 
\end{equation}
then it would contradict AH. Evaluating the above for for $\mathcal{Z}_\alpha$ with $\alpha<1$  is possible using the Landau-Gonek formula. For $\alpha > 1$ it remains an open problem. 
Now let us connect the imaginary part of $\mathcal{Z}_2$ to the problem of gaps between zeros.
Form our theorem we get that 
\begin{equation}
\label{Z_2_2}
\int \big( \Re \mathcal{Z}^{2}_2(t)- \Im  \mathcal{Z}^{2}_2(t) \big) d\mu_\lambda = \int  \big(\emph{\emph{C}}^2_2(t)-  \emph{\emph{Im}}^2_2(t)\big) d\mu_\lambda= 0.1270\cdots,
\end{equation}
Therefore a lower bound on 
$$\int \emph{\emph{Im}}^2_2(t) d\mu_\lambda,$$ would result on a lower bond on $\Re \mathcal{Z}^{2}_2.$\\


So far from \eqref{LGform} and Corollary \ref{corlambda2} we have established that $\emph{\emph{Im}}^2_2(t)$ must be small when $t$ is close to a zeta zero. \\

Now let us continue with assuming the pair correlation conjecture. From the last paragraph we see that the imaginary part and the real part of $\mathcal{Z}_2$ demonstrate some degree of independence. Therefore, we expect that $\emph{\emph{ Im}}_{2}$ to be smaller than $\emph{\emph{C}}_{2}$ inside a cluster of zeros. Having these two facts in mind, we consider our theorem for $\alpha=2$ in the following form
\begin{equation}
\label{polycI}
    0<   \int\sum_{n=0}^{[k/2]}  {k \choose 2n} (-1)^{n}\emph{\emph{C}}_{2}^{k-2n}(t) \emph{\emph{Im}}_{2}^{2n}(t) d\mu_\lambda(t) \approx \frac{1}{(k+1)!}.
\end{equation}
The higher moments in the above basically takes the problem to where the action exist, meaning that $$\mu_\lambda\big(\{t: |\emph{\emph{C}}_{2}(t)| <1 \text{ and } |\emph{\emph{Im}}_{2}(t)| <1 \}\big)$$ does not contribute much. The main contribution comes from $t$ such that $\ |\emph{\emph{C}}_{2}(t)| \geq 1 \text{ or } |\emph{\emph{Im}}_{2}(t)| \geq 1$.\\

The pair correlation conjecture suggest that $|\emph{\emph{C}}_{2}(t)|$ gets arbitrary large. Considering this, $1/(k+1)!$ in the RHS of \eqref{polycI} suggests a complicated cancellation inside the sum in \eqref{polycI}. Moreover since the expression involve higher moments of $\emph{\emph{C}}_{2}$ and $ \emph{\emph{ Im}}_{2}$ the main contribution shifts mainly from average size gaps to the zero clusters.  Explaining \eqref{polycI}, assuming PCC, can be a step forward in understating the problem in a better way.
\section{\bf Examples of Similar Random Variables} \label{sec-prob}
Motivated by our discussion in the last section and the exponential decay of moments in \eqref{polycI} we ask the following question: If pseudo moments of a complex random variable decrease very fast, would it imply that the random variable is bounded?\\

We present two random variables that are examples of negative answers to the above question.  An important property of the first one is that, our method on estimating pseudo moments does not reveal any useful information. In other words, it is an example of where our method fails.
Moments of the second  RV matches the moments of $\mathcal{Z}_2$ as in \eqref{polycI}. \\

Consider $$\mathcal{X}(t)= \sum_{n=1}^{\infty} a_n e^{int} + \tfrac{1}{2},$$
where $t \sim \mathcal{U}(0, 2\pi).$
Then $E[\mathcal{X}^k]= \tfrac{1}{2^k}.$ This example shows that not much can be said about $\mathcal{X},$ since the non-constant part that determines the behaviour of $\mathcal{X,}$ has no role in determining the psudo moments. 
In our case with $\mathcal{Z}_\alpha,$ we saw that the $k$-th pseudo moment decrease approximately like $1/k!.$ A RV like  $\mathcal{Z}_\alpha$ with this property is a bit more complicated. Here is one example:\footnote{This example is provided by Mateusz Wasilewski on mathoverflow. The exposition given here, is by Iosif Pinelis \cite{Iosif}.} \\

Let $Z:=XU$, where $X$ and $U$ are independent with $X>0$ almost surely. Assume that $X$ is unbounded with $E[X^k]< \infty$. Let $U=e^{iT}$ where $T$ is a RV with values in the interval $[0,2\pi)$ and its probability density function is given by 
\begin{equation}
\label{pdf}
    p(t)=\frac1{2\pi}\,\Big(1+2\sum_{n=1}^\infty a_n\cos nt\Big).
\end{equation}
We need to assume that $$2\sum_{n=1}^\infty a_n<1,$$ so that $p$ is well defined. Now for $k \in \mathbb{N}$ we have 
$$E[U^k]=E[e^{ikT}]=\int_0^{2\pi}e^{ikt}p(t) dt=a_k,$$
which by assuming $$0<a_k\sim\frac1{(k+1)!\,E[X]^k},$$ and by using the independence of $X, U$ we have that 
$$E[Z^k]\sim\frac1{(k+1)!},$$ as we desired. An important feature of this example is that the unboundedness of $Z$  comes from $X$ which is independent from $U.$ Therefore the event that $X$ attain large values should be independent of the event that $U$ gets large. \\ 

It is not clear how we can connect this example to our result on pseudo moments of $\mathcal{Z}_\alpha.$ To some extent it is true that when $t$ is close to an isolated zero,  $\mathcal{Z}_\alpha(t),$ behaves similarly to $e^{it},$ and $\mu_\lambda$ behaves in a similar way to $p(t).$ However, if $t$ is inside a zero cluster we must have $$ \mathcal{Z}_2 \sim X(t)e^{it,}$$ where $X$ can be viewed as accumulation factor. The problem with connecting this example to the moments of $\mathcal{Z}_\alpha$ is the independence condition. It not clear how $X(t)$ can be independent of $\mu_\lambda$ unless it behave like a constant.

\section{\bf Proof of the Theorem and its Corollaries}
In this section we are mainly seeking to estimate
\begin{align}
\int  \mathcal{Z}^{k}_\alpha(\tfrac{1}{2}+it) d\mu_{A} =\frac{ \notag \int \omega( \tfrac{1}{2}+it) \mathcal{Z}^{k}_\alpha(\tfrac{1}{2}+it)\Big|\sum_{n<T_0} \frac{a(n)}{n^{\tfrac{_1}{2}+it}}\Big|^2dt}{\int \omega( \tfrac{1}{2}+it)\Big|\sum_{n<T_0} \frac{a(n)}{n^{\tfrac{_1}{2}+it}}\Big|^2dt}
\end{align}
and 

\begin{align}
\int  |\mathcal{Z}_\alpha(\tfrac{1}{2}+it)|^2 d\mu_{A} = \frac{ \notag \int \omega( \tfrac{1}{2}+it) |\mathcal{Z}_\alpha(\tfrac{1}{2}+it)|^{2}\Big|\sum_{n<T_0} \frac{a(n)}{n^{\tfrac{_1}{2}+it}}\Big|^2dt}{\int \omega( \tfrac{1}{2}+it)\Big|\sum_{n<T_0} \frac{a(n)}{n^{\tfrac{_1}{2}+it}}\Big|^2dt},
\end{align}
where we set $T_0= T\log^{-2} T$ for abbreviation. Let us recall that 
\begin{equation*}\mathcal{Z}_\alpha(s):=\frac{-2}{\alpha \log T} \sum_{n<T^{\alpha}} \frac{\Lambda(n)}{n^{s}} \Big(1-\frac{\log n}{\alpha \log T}\Big),
\end{equation*} 

with \begin{equation*}
\mathcal{Z}_\alpha(\tfrac{1}{2} + it)= \emph{\emph{C}}_\alpha(t)+ i \emph{\emph{Im}}_\alpha(t).
\end{equation*}

Also the definition of the smooth weight is given by 
\begin{equation}
\label{omega-def}
\omega(\tfrac{1}{2}+it)=\frac{\log T}{\sqrt{\pi} T}\hspace{1 mm}\displaystyle{ e^{\frac{((t-T
)\log T)^2}{T^2}}},
\end{equation}
which basically is the Gaussian weight.  For the Fourier transform of $\omega$ we have 
\begin{equation}
\label{FTW0}
    \frac{1}{2 \pi i} \int_{(c)} \omega(s) x^s ds = \frac{1}{2 \pi} x^{\frac{1}{2}+iT} e^{-\frac{T^2 \log^2 x}{4 \log^ 2 T }}.
\end{equation}
We will use the above equation in the following form: Let $m$ or $n$ be smaller than $T \log^{-2}T,$ then 
\begin{equation}
\label{FTW}
     \int_{-\infty}^{\infty} \omega(\tfrac{1}{2}+it) \big(\frac{m}{n}\big)^{it} dt = \begin{cases} 1 \text{ if } m=n, \\
     0 \hspace{2 mm} \mbox{otherwise.}
     \end{cases}
\end{equation}

 We proceed with giving a proof of Theorem \ref{Thm}. Define $\lambda_{\beta_1, \beta_2}$ to be a completely multiplicative function with $$\lambda_{\beta_1, \beta_2}(p)=-1, \hspace{1 mm} \text{ for  }  T^{\beta_1}< p \leq T^{\beta_2}$$ and $\lambda_{\beta_1, \beta_2}(p)=1$ otherwise. Furthermore, let $\mu_{A_{\beta_1, \beta_2, r , \eta}}$ be the measure we build, as in \eqref{defmeasur}, using  
\begin{equation}\label{def_an}a(n)= \lambda_{\beta_1, \beta_2}(n)d_{r}(n) \big(1- \frac{\log n}{\log T}\big)^{\eta},\end{equation} and $d_r$ is the generalized divisor function. Note that we considered a more general definition for the arithmetic function in \eqref{def_an}. This obviously covers the arithmetic function in the statement of the theorem.

 \begin{proof}[Proof of Theorem \ref{Thm}]  We have that\begin{align*}\notag\mathcal{Z}^{k}_\alpha(\tfrac{1}{2}+it) = \Big(\frac{-2}{\alpha \log T}\Big)^k &\sum_{n_j<T^{\alpha}} \frac{\Lambda(n_1) \cdots \Lambda(n_k) }{(n_{1} \cdots n_{k})^{\tfrac{1}{2}+it}}\\ & \times \Big(1-\frac{\log n_1}{\alpha \log T}\Big) \cdots \Big(1-\frac{\log n_k}{\alpha \log T}\Big).\end{align*}
Therefore, by setting $  \hat{\Lambda}(n)= \Lambda(n) (1-\tfrac{\log n}{\alpha \log T})$ we get that
\begin{align}
\label{momz1}
   \int & \omega( \tfrac{1}{2}+it) \mathcal{Z}^{k}_\alpha(\tfrac{1}{2}+it)\Big|\sum_{n<T_0} \frac{a(n)}{n^{\tfrac{1}{2}+it}}\Big|^2dt  \\ \notag & =\Big(\frac{-2}{\alpha \log T}\Big)^k \sum_{\substack{n_j<T^{\alpha}\\ m, n<T_0}} \frac{\hat{\Lambda}(n_1) \cdots \hat{\Lambda}(n_k) a(m)a(n) }{\sqrt{mnn_{1} \cdots n_{k}}} \int \omega(\tfrac{1}{2}+ it) \big(\frac{m}{n n_1 \cdots n_k}\big)^{it}dt \\ \notag & =\Big(\frac{-2}{\alpha \log T}\Big)^k \sum_{\substack{n_j<T^{\alpha}\\ nn_{1} \cdots n_{k}<T_0}} \frac{\hat{\Lambda}(n_1) \cdots \hat{\Lambda}(n_k) a(n)a(n n_{1} \cdots n_{k}) }{nn_{1} \cdots n_{k}}.
\end{align}
This is obtained  using the fact that  $a(m)=0$ for $n>T_0$ and by applying \eqref{FTW} to the off-diagonal terms $m \neq n n_1 \cdots n_k$ in the expansion of the LHS of \eqref{momz1}. By inserting the definition of $a(\cdot)$ given in \eqref{def_an}, the above would simplify to 
 \begin{equation}
 \label{6.6}
    \Big(\frac{-2r}{\alpha \log T}\Big)^k \sum_{\substack{n_j<T^{\alpha}}} \frac{\lambda_{\beta_1, \beta_2}( n_{1})\hat{\Lambda}(n_1) \cdots \lambda_{\beta_1, \beta_2}( n_{k)}\hat{\Lambda}(n_k)  }{n_{1} \cdots n_{k}}  \sum_{n <\tfrac{T_0}{n_1\cdots n_k}} \frac{|d_r(n)|^2}{n}W_{n_1, \cdots , n_k}(n), 
\end{equation}
with $$W_{n_1, \cdots , n_k}(n)= (1-\tfrac{\log n}{\log T})^\eta\big(1-\tfrac{\log n}{\log T}-\tfrac{\log n_1}{\log T} \cdots -\tfrac{\log n_k}{\log T}\big)^\eta.$$ We use the following estimate for $d_r(\cdot)$ from~\cite{BMN}: 
\begin{equation}
\label{eqdr}
\sum_{m< x} \frac{d^2_{r}(m)}{m} = A_r(\log x)^{r^2} + O\big( (\log T)^{r^2-1}\big),
\end{equation} Therefore, by Stieltjes integration, the innermost sum equals to 
$$\int_{1}^{T_0/n_1 \cdots n_k} \frac{A_r r^2 (\log x)^{r^2-1}}{x}W_{n_1, \cdots , n_k}(x)dx+ O(\log^{r^2-1}T).$$ Using the prime number theorem we have the \eqref{6.6} equals to 
$$\int_{\substack{x_i< T^\alpha \\ x_1 \cdots x_k < T}} \tfrac{\lambda_{\beta_1, \beta_2}( x_{1})(1-\tfrac{\log x_1}{\alpha \log T}) \cdots \lambda_{\beta_1, \beta_2}(x)(1-\tfrac{\log x_k}{\alpha \log T})}{x_1 \cdots x_k}\int_{1}^{T_0/x_1 \cdots x_k} \tfrac{(\log x)^{r^2-1}W_{x_1, \cdots , x_k}(x)}{x}dx,$$ times $A_r r^2 \Big(\tfrac{-2r}{\alpha \log T}\Big)^k$ plus a negligible error term.
We use variables changes  $u=\tfrac{\log x}{\log T}$ and $v_i=\tfrac{\log x_i}{\log T}$ and we divide the whole thing by 
 \begin{align} \notag 
   \int & \omega(\tfrac{_1}{2} +it) \Big|  \sum_{} \frac{a(n)}{n^{\tfrac{_1}{2} +it}}\Big|^2dt= \sum_{n< T \log^{-2} T} \frac{|a(n)|^2}{n} \\ & \label{wholemess}= A_r r^2 \log^{r^2} T \int_{0}^{1} v^{r^2-1}(1-v)^{2\eta}dv +O\big( (\log T)^{r^2-1}\big).
 \end{align}
 Therefore we have that $\int  \mathcal{Z}^{k}_\alpha(\tfrac{1}{2}+it) d\mu_{A}$ equals to
 \begin{align*}
   \big(\frac{-2r}{\alpha}\big)^k  \frac{\int_{0}^{1} \int_{0}^{1-v_1}\cdots \int_{0}^{\small 1- v_2  \cdots- v_k }  Q(v_1, \cdots, v_k) u^{r^2-1} W_{v_1, \cdots , v_k}(u) du \hspace{1 mm} \small{d_{v_1} \ldots d_{v_k}} }{\int_{0}^{1} v^{r^2-1}(1-v)^{2\eta}dv},
 \end{align*} 
plus an error term of order $\log^{-1} T.$ Note that we take  $$Q(v_1, \cdots, v_k)= \lambda_{\beta_1, \beta_2}( v_{1})(1-\tfrac{v_1}{\alpha}) \cdots \lambda_{\beta_1, \beta_2}( v_{k})(1-\tfrac{v_k}{\alpha}).$$ 
 
 Now it is easy to see that \eqref{Z-mom} is a special case (with $\eta=0$ which is equivalent to  $W_{v_1, \cdots , v_k}(u)=1$) of what we showed in the above.\\
 
For $\alpha \geq 1$ and simple cases of $a(n)= \lambda(n), 1, $ or $ \lambda_2(n)$ we can get a much cleaner estimate than \eqref{Z-mom}:\\

If $a(n)= \lambda(n)$ we have that $a(n)a(np_1\cdots p_k)$ in \eqref{momz1} is $(-1)^k.$ 

If $a(n)= \lambda_2(n)$ and for even moments ($\mathcal{Z}_{\alpha}^{2k}$), we use
\begin{equation}
\label{lambda21}
    \lambda_2(n) \lambda_2(np_1 \cdots p_{2k})= (-1)^{k}.
\end{equation}  For $a(n)= \lambda_2(n)$ and odd moments we divide into two cases and we get
\begin{equation}
\label{lambda22}
\lambda_2(n) \lambda_2(np_1 \cdots p_k)=  \begin{cases} \lambda(n) &\mbox{if } k= 4r+1 \\
-\lambda(n). & \mbox{if } k=4r+3 \end{cases}\end{equation}

If we insert \eqref{lambda21} into \eqref{momz1} we get: 
\begin{align}
\label{eq-high-moments} 
(-1)^{k}\sum_{n_1\cdots n_{2k}<T^{\alpha}} \tfrac{\Lambda(n_1) \cdots \Lambda(n_{2k}) }{n_{1} \cdots n_{2k}}\big(1-\tfrac{\log n_1}{\alpha \log T}\big) \cdots \big(1-\tfrac{\log n_{2k}}{\alpha \log T}\big)\log\big(\tfrac{T_0}{n_1 \cdots n_{2k}}\big),
\end{align}
which justify $(-1)^k$ in \eqref{thmlambda2} in the statement of the theorem. \\

For $a(n)=\lambda(n)$ we end up with 
\begin{align}
\label{eq-high-moments1} (-1)^k\sum_{n_1\cdots n_k<T^{\alpha}} \tfrac{\Lambda(n_1) \cdots \Lambda(n_{k}) }{n_{1} \cdots n_{k}}\big(1-\tfrac{\log n_1}{\alpha \log T}\big) \cdots \big(1-\tfrac{\log n_{k}}{\alpha \log T}\big)\log\big(\tfrac{T_0}{n_1 \cdots n_{k}}\big).
\end{align}
 
Therefore, we will proceed to estimate \eqref{eq-high-moments1} since it covers both cases. Using  Perron's formula we have \begin{align*}\log \frac{T}{n_1 \cdots n_k} \mathds{1}_{n_1 \cdots n_k<T} = \frac{1}{2\pi i } \int_{(c)} \frac{T^s}{(n_1 \cdots n_k)^s} \frac{ds}{s^2}.\end{align*}Hence \eqref{eq-high-moments} comes to\begin{align*}\frac{1}{2\pi i } \int_{(c)} \sum_{n_j} & \frac{\Lambda(n_1) \cdots \Lambda(n_k) }{(n_{1} \cdots n_{k})^s}\Big(1-\tfrac{\log n_1}{\alpha \log T}\Big) \cdots \Big(1-\tfrac{\log n_k}{\alpha \log T}\Big) \frac{ds}{s^2}\\ & = \frac{1}{2\pi i } \int_{(c)} T^s \Big(\sum_{} \frac{\Lambda(n)}{n^{s+1}}- \frac{1}{\alpha \log T}\frac{\Lambda(n)\log n}{n^{s+1} } \Big)^k\frac{ds}{s^2} = \\ & \frac{1}{2\pi i } \int_{(c)} T^s \Big(-\frac{\zeta^\prime}{\zeta}(s+1) + \frac{1}{\alpha \log T} \Big(\big(\frac{\zeta^\prime}{\zeta}(s+1)\big)^{\prime}\Big)^k\frac{ds}{s^2}.\end{align*}We move the line of integration to $c= 1-\log^{-3/4}T$ and use the zero free region of the zeta function. We just need to consider the residue at $s=0$ in the above. We end up with\begin{align*}\frac{1}{2 \pi i} & \int T^s \Big(\frac{1}{s} - \frac{1}{\alpha \log T s^2}\Big)^k \frac{ds}{s^2}\\ & = \frac{1}{2 \pi i} \int T^s \sum_{i=0}^{k} {k \choose i }\big(\frac{-1}{\alpha \log T}\big)^{i} \frac{1}{s^{k+i+2}} ds \\ & = \frac{1}{2 \pi i} \sum_{i=0}^{k} {k \choose i }\big(\frac{-1}{\alpha \log T}\big)^{i} \int \frac{T^s}{s^{k+i+2}} \\ & = (\log T)^{k+1} \sum_{i=0}^{k} {k \choose i }\big(\frac{1}{\alpha }\big)^{i} \frac{(-1)^{i}}{(k+i+1)!}.
\end{align*}
Therefore we have
\begin{align*}
\int \omega( \tfrac{1}{2}+it) \mathcal{Z}^{k}_\alpha(\tfrac{1}{2}+it)\Big|\sum_{n<T_0} \frac{\lambda(n)}{n^{\tfrac{1}{2}+it}}\Big|^2dt= \Big(\frac{2}{\alpha}\Big)^k \sum_{i=0}^{k} {k \choose i} \frac{(-1)^{i}}{\alpha^i} \frac{\log T}{(k+i+1)!}.
\end{align*}
Finally we divide by $\log T$, to get to $d\mu_\lambda.$ For $\lambda_2$ we have the same calculation, with only difference that we must considered it for the $2k$-moment. This finishes the proof.
\end{proof}
Next we give a corollary of the theorem that we will use later.
\begin{corollary} We have that
\label{Cor2}
\begin{align}
 & \label{coreq} \int \Big({\emph{C}}_1^2(t) - \textrm{ \emph{Im}}_1^2(t)\Big) d\mu_\lambda= 0.36666\cdots + O\big(\tfrac{1}{\log T}\big), \\ & \label{core1} \int \Big({\emph{C}}_1^3(t) - 3 {\emph{C}}_1(t)\textrm{  \emph{Im}}^2(t)\Big) d\mu_\lambda= 0.16504 \cdots + O\big(\tfrac{1}{\log T}\big), \\ &  \label{coreq2} \int \Big({\emph{C}}_1^4(t) - 6 {\emph{C}}_1^2(t)\textrm{  \emph{Im}}^2(t)+ \textrm{  \emph{Im}}^4(t)\Big) d\mu_\lambda= 0.06194\cdots + O\big(\tfrac{1}{\log T}\big).
\end{align}
\end{corollary}
\begin{proof}
We have that$$\Re (\mathcal{Z}^2_\alpha(\tfrac{1}{2}+it))= C_\alpha(t)^2-\emph{\emph{Im}}^2_\alpha(t),$$and $$\Re (\mathcal{Z}^3_\alpha(\tfrac{1}{2}+it))= C_\alpha(t)^3 -3 C_\alpha(t) \text{Im}^2_\alpha(t),$$
also a similar equation for the real part of the $4$-th moment $$\Re (\mathcal{Z}^4_\alpha(\tfrac{1}{2}+it))= C_\alpha(t)^4 -6 C^2_\alpha(t)\text{Im}^2_\alpha(t) + \text{Im}^4_\alpha(t).$$  We consider the above equations and then we apply Theorem \ref{Thm} to get the corollary.\end{proof}

We continue with giving a proof of Corollaries \ref{cor 01} and \ref{cor 1}.  Using \eqref{Z-def} we have  $$\emph{\emph{C}}_{\alpha}(t) = \Re \mathcal{Z}_\alpha =\frac{-2\Re}{\alpha \log T} \sum_{n<T^{\alpha}} \frac{\Lambda(n)}{n^{\tfrac{_1}{2}+ it}} \Big(1-\frac{\log n}{\alpha \log T}\Big).  $$
Therefore,
\begin{equation}
\label{3.11}
    \int \emph{\emph{C}}^{2}_\alpha(t) d\mu_\lambda = \int \frac{ (\mathcal{Z}_\alpha+ \overline{\mathcal{Z}}_\alpha)^{2}}{4} d\mu_\lambda = \int  \frac{\mathcal{Z}^{2}_\alpha}{4} + \frac{\overline{\mathcal{Z}^{2}_\alpha}}{4} + \frac{|\mathcal{Z}_\alpha|^2}{2} d\mu_\lambda.
\end{equation}
Note that\begin{equation}
\label{appthm1}
    \int \frac{\mathcal{Z}^{2}_\alpha}{4} + \frac{\overline{\mathcal{Z}^{2}_\alpha}}{4} =  \int \frac{\emph{\emph{C}}_1(t)^2 - \textrm{ \emph{\emph{Im}}}_1^2(t)}{2} d\mu_\lambda= 0.1833333\cdots,
\end{equation}
using Corollary \ref{Cor2}. For estimating
\begin{equation*}
    \int |\mathcal{Z}_\alpha|^2 d\mu_A, 
\end{equation*}
we use the following lemma.
\begin{lemma}
\label{lemmsquare} Let $a(n)= \lambda(n)d_{r}(n) \big(1- \frac{\log n}{\log T}\big)^{\eta},$
and set 
\begin{equation}
\label{sumbm}
\mathcal{Z}_\alpha(s) \big(\displaystyle{ \sum_{n < T\log^{-2}T}} a(n)n^{-s}\big)= \sum_{m<T^{\alpha+1}} \frac{b(m)}{m^{s}}
\end{equation}
Then we have 
\begin{align}
\label{5.15sq}
 \int |\mathcal{Z}_\alpha(& \tfrac{1}{2}+it)|^2 d\mu_{A_{r, \eta }}=  O\big(\frac{1}{\log T}\big ) \notag \\ & + \frac{\sum_{m< T\log^{-2}T} \frac{|b(m)|^2}{m}+  \int \omega(\tfrac{1}{2} + it)\bigg|\displaystyle{\sum_{{T}{\small{\log^{-2}T}} \leq n}} \frac{b(n)}{n^{\tfrac{_1}{2}  + it}}\bigg|^2 }{A_r \log^{r^2} T \int_{0}^{1} v^{r^2-1}(1-v)^{2\eta}dv}.
\end{align}
Moreover, we have
\begin{align}
 \sum_{m< T\log^{-2}T} & \frac{|b(m)|^2}{m}    = \label{lemsq}
     A_r\frac{4r^2}{\alpha^2}  \log^{r^2} T  {\int_{0}^{1}\int_{0}^{1-v} v Q(u, v, v)du dv} \\ & + \notag A_r\frac{4r^4}{\alpha^2}  \log^{r^2} T 
   {\displaystyle{\int_{0}^{1}} Q(u, v_1, v_2)du dv_2dv_1 } + O\big(\log^{r^2-1} T\big),  
\end{align}
with 
\begin{equation*}
    Q(u, v_1, v_2) = u^{r^2-1} \big(1- \frac{v_1}{\alpha}\big)\small{\big(1- \frac{v_2}{\alpha}\big)} \small{(1-u -v_1)^\eta (1-u -v_2)^\eta}.
\end{equation*}
\end{lemma}

\begin{proof}
First note that the denominator in \eqref{5.15sq} comes from the fact that we considered $ d\mu_{A_{ r, \eta }}$ to be a probability measure, hence we need to divide by $\sum |a(n)|^2/n$, which using \eqref{wholemess} is $$ A_r \log^{r^2} T \int_{0}^{1} v^{r^2-1}(1-v)^{2\eta}dv + O(\log^{r^2-1} T ).$$ For the sum with $b(n)$ in \eqref{sumbm}, we split the sum
\begin{align*}
     \sum_{m<T^{\alpha+1}} \frac{b(n)}{n^{s}}=  \sum_{m<T \log^{-2}T} \frac{b(m)}{m^{s}} +  \sum_{T \log^{-2}T \leq n<T^{\alpha+1}} \frac{b(n)}{n^{s}}.
\end{align*}
Again, for simplicity we set $T_0= T\log^{-2} T.$ Hence

\begin{align*}
  &  \int  \omega(\tfrac{1}{2}  + it) \bigg|\sum_{m<T^{\alpha+1}} \frac{b(m)}{m^{\tfrac{1}{2} + it}}\bigg|^2=   \int \omega(\tfrac{1}{2} + it)  \bigg|\sum_{m<T_0} \frac{b(m)}{m^{\tfrac{1}{2}  + it}}\bigg|^2   \\ & +  \int \omega(\tfrac{1}{2} + it)\bigg|\sum_{T_0 \leq n<T^{\alpha+1}} \frac{b(n)}{n^{\tfrac{1}{2}  + it}}\bigg|^2 + 2 \Re \sum_{\substack{m<T_0 \\ T^{\alpha} > n \geq T_0 }} \frac{b(n)b(m)}{\sqrt{mn}} \int \omega(\tfrac{1}{2} + it) \big(\frac{m}{n}\big)^{it}
\end{align*}
The last integral using the Fourier transform of the Gaussian in \eqref{FTW0} is \begin{equation}
\label{gaussian1}
    \ll  e^{- \frac{T^2 \log^2 (m/n)}{4\log^2 T}}.
    \end{equation}
Using the fact that $m<T_0$  and  $n \geq T_0$ we have $$|\log \big( \frac{m}{n} \big)|= |\log \big(1+ \frac{n-m}{m}\big)| > \frac{1}{m}> \frac{\log^2 T}{T}.$$
Therefore \eqref{gaussian1} is smaller than  $e^{- \log^2 T},$ and therefore, using $b(m)< m^{\epsilon},$ we get   $$ \frac{b(n)b(m)}{\sqrt{mn}} \int \omega(\tfrac{1}{2} + it) \big(\frac{m}{n}\big)^{it} \ll  T^{\alpha/2 + \epsilon} e^{- \log^2 T} \ll T^{-B}, $$ for every $B>0.$ Therefore we established that 
\begin{equation*}
   \int |\mathcal{Z}_\alpha|^2 d\mu_{A_{ r, \eta }} = \frac{\sum_{n< T_0} \frac{|b(m)|^2}{m} + \int \omega(\tfrac{1}{2} + it)\bigg|\sum_{T_0 \leq \hspace{1 mm} n <T^{\alpha+1}} \frac{b(n)}{n^{\tfrac{_1}{2}  + it}}\bigg|^2dt}{ \sum_{n < T_0} \frac{|a(n)|^2}{n}} + O(\tfrac{1}{T^B}),
\end{equation*}
which justifies \eqref{5.15sq}. To show \eqref{lemsq}, for $n<T_0,$ we have
\begin{align}
    b(n)= \notag & - \frac{2}{\alpha \log T} \sum_{d|n}\Lambda(d)\big(1-\frac{\log(d)}{\alpha \log T}\big) \lambda(n/d) d_r(n/d) \big(1 -\frac{\log n}{\log T} +\frac{\log d}{\log T} \big)^\eta \\ & = -\frac{2}{\alpha \log T} d_r(n)\lambda(n) \sum_{p^h||n} s_h(p) \log p,
\end{align}
where $p^h || n$ means that $p^h|n$ and $p^{h+1} \nmid n$ and  $$s_h(p)= \sum_{j=1}^{h}(-1)^j\Big(\frac{\Gamma(h-j+r)h!}{\Gamma(h+r)(h-i)!}\Big) \big(1- \frac{j\log p}{\alpha \log T}\big)\big(1- \frac{\log n}{ \log T}+ \frac{j\log p}{ \log T}\big)^\eta.$$
We used the following formula 
\begin{equation*}
d_r(p^k)= \frac{\Gamma(k+r)}{\Gamma(r)k!},
\end{equation*}
in the above. Consequently, we have
\begin{align}
\label{3.16}
    & \notag \sum_{n< T_0} \frac{|b(n)|^2}{n}= \frac{4}{(\alpha \log T)^2} \sum_{n< T_0}  \frac{\big(d_r(n)\sum_{p^h||n} s_h(p) \log p\big)^2}{n}  \\  & = \frac{4}{(\alpha \log T)^2} \sum_{n< T_0} \frac{d^2_r(n)\sum_{p^h||n} s_h(p) \log p\sum_{q^{h^{\prime}}||n} s_{h'}(q) \log q}{n}.
\end{align}
We show that the contribution of $h> 1$ or $h^{\prime} >1$ is negligible. We have that \eqref{3.16} is smaller than 
$$ \frac{1}{(\alpha \log T)^2} \sum_{p^h, q^h{^{\prime}} < T_0} \frac{\log p  \log q}{r^2p^h q^{h^{\prime}}} \sum_{\substack{ p \nmid n , q \nmid n \\ n< T/(p^r q^r{^{\prime}})} } \frac{d^2_r(n)}{n}$$
We use \eqref{eqdr} and we obtain that the above is $$ \ll \log^{r^2-1} T,$$ if either of $h, h'$ are $>1.$ Therefore considering the fact that the total mass of $ \mu_{A_{r, \eta }}$ in the denominator in \eqref{lemsq} is of size $\log^{r^2} T$, proves the contribution of $h$ or $h'$ bigger than $1$ is negligible.  Hence,  we continue with $h=h^{\prime}=1.$ We set $n=mpq$  in \eqref{3.16}, when $p \neq q$,  and  by changing the order of the summation we are seeking to estimate 
\begin{align}\notag
\frac{4}{(\alpha \log T)^2} & \sum_{p, q <T_0} \frac{ \log p \big(1- \tfrac{\log p}{\alpha \log T} \big)  \log q \big(1- \tfrac{\log q}{\alpha \log T}\big)}{r^2pq} \\ & \times \sum_{\substack{ p \nmid m , q \nmid m \\ m< T/p q}} \frac{d^2_r(mpq)\big(1- \frac{\log m}{ \log T}+ \frac{\log p}{ \log T}\big)^\eta \big(1- \frac{\log m}{ \log T}- \frac{\log q}{ \log T}\big)^\eta}{m}.    
\end{align} 
Therefore, for $p \neq q,$ we have
\begin{align}
\label{cacl1}
  \frac{4r^2}{(\alpha \log T)^2} \sum_{pq < T_0} & \frac{ \log p \big(1- \tfrac{\log p}{\alpha \log T} \big)  \log q \big(1- \tfrac{\log q}{\alpha \log T} \big)}{pq}   \\ & \times  \sum_{\substack{ p \nmid m , q \nmid m \\ m< T/pq }}\notag\frac{d^2_r(m) \big(1- \frac{\log m}{ \log T}- \frac{\log p}{ \log T}\big)^\eta \big(1- \frac{\log m}{ \log T}- \frac{\log q}{ \log T}\big)^\eta}{m}.
\end{align}
For $p=q,$ we set $m=np$ in \eqref{3.16} and we get  
\begin{equation}
\label{cacl2}
  \frac{4}{(\alpha \log T)^2} \sum_{p < T_0}  \frac{ \log^2 p \big(1- \tfrac{\log p}{\alpha \log T} \big)^2  }{p}  \sum_{\substack{ p \nmid m  \\ m< T/p }} \frac{d^2_r(m)}{m} \big(1- \frac{\log m}{ \log T}\big)^{2\eta}.
 \end{equation}
 
Conditions $p \nmid m$ and $p \nmid m, q \nmid m$ in \eqref{cacl1}
 and \eqref{cacl2} can be dropped with a negligible error term, since we have $$\sum_{m < T, \hspace{1 mm} p|n } \frac{1}{m}= \frac{1}{p}\log(T/p) + O(1/T). $$
Applying the prime number theorem  and \eqref{eqdr} imply that \eqref{cacl1} equals to 
\begin{equation}
\label{3.20}
A_r\frac{4r^4}{\alpha^2}  \log^{r^2} T   \int_{0}^{1}  \int_{0}^{1-v_1}\int_{0}^{1-v_1-v_2} Q(u, v_1, v_2) dudv_1 dv_2 + O(\log^{r^2-1} T),
\end{equation}
Recall that $$Q(u, v_1, v_2) = u^{r^2-1} \big(1- \frac{v_1}{\alpha}\big)\small{\big(1- \frac{v_2}{\alpha}\big)} \small{(1-u -v_1)^\eta (1-u -v_2)^\eta}.$$ Similarly,  \eqref{cacl2} equals to 
\begin{equation*}
   A_r\frac{4r^2}{\alpha^2} \log^{r^2} T \int_{0}^{1}\int_{0}^{1-v} u^{r^2-1}  \small{(1-u -v)^{2\eta}v \big(1- \frac{v}{\alpha}\big)^2 } dudv + O(\log^{r^2-1} T).
\end{equation*}
This completes the proof of Lemma \ref{lemmsquare}.
\end{proof}
\begin{proof}[Proof of Corollaries \ref{cor 01} and \ref{cor 1}]
 By Lemma \ref{lemmsquare}  for $r=1$ and $\eta=0,$ we have that
\begin{align}
\label{cor14}
      \int |\mathcal{Z}_\alpha|^2 d\mu_\lambda \geq &  4\alpha^{-2}  \int_{0}^{1} \int_{0}^{1-u} (1-u -v)\big(1- \frac{u}{\alpha}\big)\big(1- \frac{v}{\alpha}\big)dv du   \\ & \notag + 4\alpha^{-2} \int_{0}^{1} u(1-u)\big(1- \frac{u}{\alpha}\big)^2 du  + O(\log^{-1} T) \\ & = \frac{4}{3\alpha^2} - \frac{1}{\alpha^3} + \frac{7}{30\alpha^4}.
\end{align}
For $\alpha=1,$ we get 
\begin{equation}
\label{3.23}
    \int |\mathcal{Z}_1|^2 d\mu_\lambda \geq 0.56664.
\end{equation}
Recall that by \eqref{3.11} we have $$\int \emph{\emph{C}}^{2}_1(t) d\mu_\lambda =  \int  \Big(\frac{\mathcal{Z}^{2}_1}{4} + \frac{\overline{\mathcal{Z}^{2}_1}}{4} + \frac{|\mathcal{Z}_1|^2}{2}\Big) d\mu_\lambda.$$
By combining \eqref{appthm1} and \eqref{3.23} we have 

\begin{equation}
\label{VarC1}
    \int \emph{\emph{C}}^{2}_1(t) d\mu_\lambda  \geq 0.46666.
\end{equation}
Using $$\int \Big(\emph{\emph{C}}_1(t)^2 -  \emph{\emph{Im}}_1^2(t)\Big) d\mu_\lambda= 0.36666\cdots,$$  we get a lower bound for $ \emph{\emph{Im}}_1^2.$
For the $4$-th moment note that 
 
$$\int \big( \emph{\emph{C}}_1^{2}(t)- 0.46666)^2  d\mu_\lambda(t)> 0,$$ which gives \begin{equation}
\label{three}    
\int \emph{\emph{C}}_1^{4}(t) d\mu_\lambda(t)  >2\times 0.46666\int  \emph{\emph{C}}_1^{2}(t)  d\mu_\lambda(t)  - (0.46666)^2 > (0.46666)^2.
\end{equation}
Similarly we can get a lower bound for the expectation of $\emph{\emph{Im}}_1^{4}.$\\

For proving \eqref{two} we use \eqref{three} and the lower bound we get for the $4$-th power of  $\emph{\emph{Im}}_1,$ and insert them in $$\int \Big(\emph{\emph{C}}_1^4(t) - 6 \emph{\emph{C}}_1^2(t)\emph{  \emph{Im}}_1^2(t)+ \emph{  \emph{Im}}_1^4(t)\Big) d\mu_\lambda= 0.06194\cdots + O\big(\tfrac{1}{\log T}\big),$$ which we obtained in Corollary \ref{Cor2}. This finishes the proof. 
\end{proof}
\begin{proof}[Proof of Corollary \ref{CorMain}] For $r=2$ and $\eta=0,$ using the theorem we have that 
\begin{equation}
\label{5.32}
    \int  \big(\emph{\emph{C}}^2_1(t)-  \emph{\emph{Im}}^2_1(t)\big) d\mu_{A_{0, 1, 2}}= 0.39047\cdots,
\end{equation} 
where $\mu_{A_{0, 1, 2}}$ correspond to the measure we have with $a(n)=\lambda(n)d_2(n).$ On the other hand we have that 
$$\int  \emph{\emph{C}}_1(t) d\mu_{A_{0, 1, 2}}= 2/3,$$ which gives that $$\int \emph{\emph{C}}^2_1(t) d\mu_{A_{0, 1, 2}} > 4/9.$$ Employing \eqref{5.32} we have that $$\int \emph{\emph{Im}}^2_1(t) d\mu_{A_{0, 1, 2}} > 0.0539,$$ and therefore
\begin{equation}
\label{5.33}
    \int |\mathcal{Z}_1|^2(t) d\mu_{A_{0, 1, 2}} > 0.498.
\end{equation}

Using Lemma \ref{lemmsquare} we have that 
 $$\int  |\mathcal{Z}_1|^2 d\mu_{A_{0, 1, 2}}= 0.4619+  \frac{  \int \omega(\tfrac{1}{2}+it)\bigg|\displaystyle{\sum_{{T}{\small{\log^{-2}T}} \leq n}} \frac{b(n)}{n^{\tfrac{_1}{2}  + it}}\bigg|^2dt }{A_2 \log^{4} T \int_{0}^{1} v^{3}dv},$$
 where $b(n)= -\frac{2}{ \log T} \sum_{d|n}\Lambda(d)\big(1-\frac{\log(d)}{ \log T}\big) \lambda(n/d) d_2(n/d) \asymp \lambda(n)d_2(n)\frac{ \log n}{ \log T}.$ Now by using \eqref{5.33}, we have that 
 \begin{equation*}
     \int \omega(\tfrac{1}{2}+it)\bigg|\displaystyle{ \sum_{{T}{\log^{-2}T}\leq n} } \hspace{2 mm} \displaystyle{ \frac{b(n)}{n^{\tfrac{_1}{2}  + it}}}\bigg|^2dt > 0.009 A_2 \log^{4} T. 
 \end{equation*}
 
\end{proof}
\begin{proof}[ Corollary \ref{corlambda2}]
The proof goes similarly to the proof of Corollary \ref{cor 1}, therefore we only comment on parts that are different. We have 
\begin{equation}
\label{prooflambda21}
    \int \emph{\emph{Im}}^{2}_\alpha(t) d\mu_{\lambda_2}  = \int  \frac{|\mathcal{Z}_\alpha|^2}{2} -\Big(\frac{\mathcal{Z}^{2}_\alpha}{4} + \frac{\overline{\mathcal{Z}^{2}_\alpha}}{4}\Big) \hspace{1 mm} d\mu_{\lambda_2}.
\end{equation}
Using \eqref{thmlambda2} with $k=1$ we get that $$\int \Big(\frac{\mathcal{Z}^{2}_\alpha}{4} + \frac{\overline{\mathcal{Z}^{2}_\alpha}}{4}\Big) d\mu_{\lambda_2}= -0.18333\cdots .$$
For $\int {|\mathcal{Z}_\alpha|^2}d\mu_{\lambda_2},$ we use Lemma \ref{lemmsquare} with 
\begin{align}
\label{lambda2eq2}
    b(n)=  - \frac{2}{\alpha \log T} \sum_{d|n}\Lambda(d)\big(1-\frac{\log(d)}{\alpha \log T}\big) \lambda_2(n/d) 
\end{align}
As we seen with the proof of Corollary \ref{cor 1}, in estimating  $$\sum_{n< T\log^{-2}T} \frac{|b(m)|^2}{m}$$ main contributions   comes form $d=p$ and therefore $\lambda_2(n/d)$ can be pulled out of the sum in   \eqref{lambda2eq2} with a factor of $$(-1)^{\big[(\Omega(n)-1)/2\big]}.$$ Since we have in the sum of $|b(m)|^2,$ we can ignore the contribution of $-1,$ and the rest of calculation is the same as the calculation for $\lambda$ in \eqref{3.16}.
\end{proof}
\begin{proof}[Proof of Corollary \ref{Cor14}] Equation \eqref{cor14} basically gives the result:
 $$\int |\mathcal{Z}_\alpha|^2 d\mu_\lambda \geq   \frac{4}{3\alpha^2} - \frac{1}{\alpha^3} + \frac{7}{30\alpha^4}.$$ Now let us compare this with what we get from only applying the Mollifier method. Using the theorem we have that \begin{equation*}
   \int  \Re (\mathcal{Z}^2_\alpha(\tfrac{1}{2}+it)) d\mu_\lambda= \int \big(C_\alpha(t)^2-\emph{\emph{Im}}^2_\alpha(t)\big) d\mu_\lambda= \tfrac{4}{6}\alpha^{-2}- \tfrac{1}{3}\alpha^{-3} + \tfrac{1}{30}\alpha^{-4}.
\end{equation*}
On the other hand for $\alpha >1,$ using the mollifier method, we have that 
$$\int C_\alpha(t)^2 d\mu_\lambda > \big(\frac{1}{\alpha} - \frac{1}{3\alpha^2}\big)^2,$$ therefore  
$$\int \emph{\emph{Im}}^2_\alpha(t) d\mu_\lambda > \frac{1}{3\alpha^2} - \frac{1}{3\alpha^3} - \frac{17}{90 \alpha^4},$$
which give $$ \int  \big(C_\alpha(t)^2+ \emph{\emph{Im}}^2(t) \big) d\mu_\lambda \geq \frac{4}{3}\alpha^{-2}- \alpha^{-3} + \tfrac{17}{90}\alpha^{-4}.$$

\end{proof}
\section*{Acknowledgement} This work was supported by Institute for Research in Fundamental Sciences (IPM) in Iran. I would like to thank them for their support and hospitality during my visit.

\end{document}